\documentclass[preprint,review,10pt]{amsart}
\usepackage{amsmath,amssymb,amsfonts,xspace}
\newtheorem{theorem}{Theorem}[section]

\theoremstyle{definition}
\newtheorem{definition}[theorem]{Definition}
\newtheorem{example}[theorem]{Example}

\newtheorem{corollary}[theorem]{Corollary}

\theoremstyle{remark}

\numberwithin{equation}{section}

\begin{document}

\title{ Interval valued $(\alpha,\beta)$-fuzzy hyperideals in  Krasner $(m,n)$-hyperrings
  }

\author{M. Anbarloei}
\address{Department of Mathematics, Faculty of Sciences,
Imam Khomeini International University, Qazvin, Iran.
}

\email{m.anbarloei@sci.ikiu.ac.ir }





\begin{abstract}
In this paper, the notion of quasicoincidence of a fuzzy interval valued with an interval valued fuzzy set, which generalizes the concept of quasicoincidence of a fuzzy point in a fuzzy set is concentrated. Based on the idea, 
we  study the concept of  interval valued $(\alpha, \beta)$-fuzzy hyperideals  in Krasner $(m,n)$-hyperrings.  In particular, some
fundamental aspects of interval valued   $(\in, \in \vee q)$-fuzzy hyperideals will be considered. Moreover,  we examine  the notion of implication-based interval valued fuzzy hyperideals in a Krasner $(m,n)$-hyperring.

\end{abstract}
\maketitle
\section{Introduction}
In this section, we describe the motivation and a survey of
related works.  The concept of  fuzzy sets was proposed  by Zadeh in 1965 \cite{zadeh}.  Fuzzy set theory is a useful tool to describe situations in which the data are imprecise or vague. Fuzzy sets handle such situations by attributing a degree to which a certain object belongs to a set. After the pioneering work of Zadeh, there are many papers devoted to fuzzify the classical mathematics into fuzzy mathematics.  
 Rosenfeld introduced and studied fuzzy sets in the context of group theory and formulated the notion of a fuzzy subgroup of a group \cite{Rosenfeld}. The concepts of fuzzy subrings and ideals were considered by in Liu \cite{liu}.  
 
 After the introduction of the notion of hypergroups as a generalization of groups by  Marty \cite{s1} in 1934, many papers and books concerning hyperstructure theory have appeared in literature. A short review of this theory appears in  \cite {s2, s3, davvaz1, davvaz2, s4}. The notion of n-group is another generalization of groups.  The idea of investigations of $n$-ary algebras goes back to Kasner’s lecture \cite{s5} at the 53rd annual meeting of the American Association of the Advancement of Science in 1904.    Dorente wrote the first paper concerning the theory of $n$-ary groups  about 90 years ago \cite{s6} . Later on, Crombez and
Timm \cite{s7, s8} defined  the notion of the $(m, n)$-rings and their quotient structures.  
 n-ary generalization of algebraic structures is the most
natural way for further development and deeper understanding
of their fundamental properties  \cite{l1, l2, l3, ma, rev1, s9,cons}.
One important class of hyperrings  is called the Krasner hyperring \cite{kras}.  In \cite{d1},  a generalization of the Krasner hyperrings, which is a subclass of $(m,n)$-hyperrings, was defined by Mirvakili and Davvaz. It is called Krasner $(m,n)$-hyperring.  Ameri and Norouzi in \cite{sorc1} introduced some important
classes of hyperideals in this respect. Many other interesting papers have been written  on Krasner $(m,n)$-hyperring  \cite{mah2, mah3, asadi, rev2,  d1, nour,   rev1, Yassine}.

Many researchers are interested in fuzzy hyperstructures because
of nice connection between fuzzy sets and hyperstructures  \cite{Zahedi1, Zahedi2, Leoreanu1, Corsini, zhan2, Yin, asli}.  Zhan and et. al. in \cite{zhan} concentrated on the quasicoincidence of a fuzzy interval value with an interval valued fuzzy set. Davvaz in \cite{davvaz2} introduced the notion of a fuzzy hyperideal of a Krasner $(m,n)$-hyperring
and to extended the fuzzy results to Krasner $(m,n)$-hyperring.
 Interval valued fuzzy sets were introduced independently by Zadeh \cite{zadeh2}. Zhan and et. al.\cite{zhan} concentrated on the quasi-coincidence of a fuzzy interval value with an interval valued fuzzy set. The notion of
quasicoincidence of a fuzzy interval valued with an interval valued
fuzzy set, which generalizes the concept of
quasicoincidence of a fuzzy point in a fuzzy set, was introduced in \cite{davvaz3}. The properties of the interval valued $(\alpha, \beta)$-fuzzy hyperideals in a semihyperring were studied in \cite{Ahmed}. This concept was generalized by Li and et. al. in \cite{Li}. The notion of interval valued $(\alpha,\beta)$-fuzzy subalgebraic hypersystems in an algebraic hypersystem, which is a generalization of a
fuzzy subalgebraic system, was defined in \cite{zhan3}. Dutta in  \cite{Dutta} established the concept of Interval valued fuzzy prime and semiprime ideals of a hypersemiring.

In this paper, our aim is to consider the concept of
quasicoincidence of a fuzzy interval valued with an interval valued
fuzzy set, which generalizes the concept of quasicoincidence of a fuzzy point in a fuzzy set. This paper is organized as follows. In section 2, we recall some terms and definitions which we need to develop our paper.
In Section 3, we  analyze  entropy of interval valued $(\alpha, \beta)$-fuzzy hyperideals  in Krasner $(m,n)$-hyperrings. In Section 4, some
fundamental aspects of interval valued   $(\in, \in \vee q)$-fuzzy hyperideals have been investigated. Finally, in Section 5, we discuss  the notion of implication-based interval valued fuzzy hyperideals of Krasner $(m,n)$-hyperrings.

\section{Preliminaries}
In this section , we recall some basic notions and results of fuzzy algebra and  Krasner $(m,n)$-hyperrings
which we shall use in this paper.\\
 A fuzzy subset of $G$ is a function  $\mu : G \longrightarrow L$ such that $L$ is the unit interval  $[0,1] \subseteq \mathbb{R}$. The set of all fuzzy subsets of $G$ is denoted by $L^G$ .  The set, $\{a \in G  \ \vert \ \mu(a) \neq 0
 \}$ is called the support of $\mu$ and is denoted by $supp(\mu)$. 

\begin{definition}
\cite{davvaz2} A Fuzzy subset $\mathcal{A}$ of a Krasner $(m,n)$-hyperring $\mathcal{R}$ is said to be a fuzzy hyperideal of $\mathcal{R}$ if the following conditions hold:

(i) $min \{\mathcal{A}(a_1),\dots,\mathcal{A}(a_m)\} \leq inf\{f(c) \vert \ c \in f(a_1^m)\}$ for all $a_1^m \in \mathcal{R}$;

(ii) $\mathcal{A}(a) \leq \mathcal{A}(-a)$, for all $a \in \mathcal{R}$;

(iii) $max\{\mathcal{A}(a_1),\dots,\mathcal{A}(a_n)\} \leq \mathcal{A}(g(a_1^n))$, for all $a_1^n \in \mathcal{R}$.

\end{definition}
It was shown (Theorem 5.6 in \cite{davvaz2}) that a fuzzy subset $\mathcal{A}$ of a Krasner $(m,n)$-hyperring $\mathcal{R}$ is a fuzzy hyperideal if and only if every its non-empty level subset
is a hyperideal of $\mathcal{R}$.

\begin{definition} \cite{zadeh2}
An interval number on $[0, 1]$, denoted by $\tilde{x}$, is deﬁned as the closed subinterval of $[0, 1]$, where $\tilde{x}=[x^-,x^+]$ satisfying $0 \leq x^- \leq x^+ \leq 1$.
\end{definition}
$D[0,1]$ denotes the set of all interval numbers. The interval $[x,x]$ can be simply identified by the number $x$. 
Let $\tilde{x}_i=[x_i^-,x_i^+], \tilde{y}_i=[y_i^-,y_i^+] \in D[0,1]$ for $i \in I$. Then we define:

$rmin\{\tilde{x}_i,\tilde{y}_i\}=[min\{x_i^-,y_i^-\},min\{x_i^+,y_i^+\}],$

$rmax\{\tilde{x}_i,\tilde{y}_i\}=[max\{x_i^-,y_i^-\},max\{x_i^+,y_i^+\}],$

$rsup\ \tilde{x}_i=[\bigvee_{i \in I}x_i^-, \bigvee_{i \in I}x_i^+],$

$rinf\ \tilde{x}_i=[\bigwedge_{i \in I}x_i^-, \bigwedge_{i \in I}x_i^+],$\\
and put

$(1) \ \tilde{x}_1 \leq \tilde{x}_2 \Longleftrightarrow x_1^- \leq x_2^- \ \text{and} \  x_1^+ \leq x_2^+$,

$(2) \  \tilde{x}_1 = \tilde{x}_2 \Longleftrightarrow x_1^- = x_2^- \ \text{and} \  x_1^+ = x_2^+$,

$(3) \  \tilde{x}_1 < \tilde{x}_2 \Longleftrightarrow \tilde{x}_1 \leq \tilde{x}_2 \ \text{and} \ \tilde{x}_1 \neq \tilde{x}_2$,

$(4) \ k\tilde{x}=[kx^-,kx^+]$ for $0 \leq k \leq 1$.\\
Clearly, $(D[0,1],\leq , \wedge,\vee)$ forms a complete
 lattice with the least element $0=[0,0]$ and the greatest element $1=[1,1]$.\\
Recall from \cite{zadeh2} that an interval valued fuzzy subset $\mathcal{A}$ on $X$ is the
set 
\[\mathcal{A}=\{(x,[\tilde{\mu}^-_{\mathcal{A}}(x), \tilde{\mu}^+_{\mathcal{A}}(x)]) \vert x \in X\},\]
such that $\tilde{\mu}^-_{\mathcal{A}}$ and $\tilde{\mu}^+_{\mathcal{A}}$ are two fuzzy subsets of $X$ with $\tilde{\mu}^-_{\mathcal{A}}(x) \leq \tilde{\mu}^+_{\mathcal{A}}(x)$
for all $x \in X$. We put $\tilde{\mu}_{\mathcal{A}}(x)=[\tilde{\mu}^-_{\mathcal{A}}(x), \tilde{\mu}^+_{\mathcal{A}}(x)]$. Then we have $\mathcal{A}=\{(x,\tilde{\mu}_{\mathcal{A}}(x)) \vert x \in X\}$ such that $\tilde{\mu}_{\mathcal{A}}: X \longrightarrow D[0,1]$.

An interval valued fuzzy set $\mathcal{A}$ of a Krasner $(m,n)$-hyperring $\mathcal{R}$ of the form
\[
 \tilde{\mu}_\mathcal{A}(y)=\left\{
 \begin{array}{lr}
 \tilde{s}(\neq[0,0]) &\text{if $y=x$,}\\
\lbrack 0,0 \rbrack &\text{ otherwise.}
 \end{array} \right.\]
 is called a fuzzy interval value with support $x$ and interval value $\tilde{s}$ and is denoted by $F(x;
\tilde{s})$. A fuzzy interval value $F(x;
\tilde{s})$ is said to belong to  (resp. be quasi-coincident with) an interval valued fuzzy set $\mathcal{A}$, written as $F(x;
\tilde{s}) \in \mathcal{A}$ (resp. $F(x;
\tilde{s}) q \mathcal{A}$) if $\tilde{\mu}(_\mathcal{A}(x) \geq  \tilde{s}$ ( resp. $\tilde{\mu}_\mathcal{A}(x) +\tilde{s} > [1,1]$). We write $F(x;
\tilde{s}) \in \vee q $ (resp. $F(x;
\tilde{s}) \in \wedge q $) $\mathcal{A}$ if $F(x;
\tilde{s}) \in \mathcal{A}$ or (resp. and) $F(x;
\tilde{s}) q \mathcal{A}$. If $\in \vee q$ does not hold, then we write $\overline{\in \vee q}$.

Suppose that  $G$ is a nonempty set. $P^*(G)$ denotes  the 
set of all the nonempty subsets of $G$. The map $f : G^n \longrightarrow P^*(G)$
is called an $n$-ary hyperoperation and the algebraic system $(G, f)$ is called an $n$-ary hypergroupoid. For non-empty subsets $G_1,..., G_n$ of $G$ we define
$f(G^n_1) = f(G_1,..., G_n) = \bigcup \{f(a^n_1) \ \vert \ a_i \in G_i, i = 1,..., n \}$.
The sequence $a_i, a_{i+1},..., a_j$ 
will be denoted by $a^j_i$. For $j< i$, $a^j_i$ is the empty symbol. Using this notation,
$f(a_1,..., a_i, b_{i+1},..., b_j, c_{j+1},..., c_n)$
will be written as $f(a^i_1, b^j_{i+1}, c^n_{j+1})$. The expression will be written in the form $f(a^i_1, b^{(j-i)}, c^n_{j+1})$, when $b_{i+1} =... = b_j = b$ . 
If for every $1 \leq i < j \leq n$ and all $a_1, a_2,..., a_{2n-1} \in G$, 

$f(a^{i-1}_1, f(a_i^{n+i-1}), a^{2n-1}_{n+i}) = f(a^{j-1}_1, f(a_j^{n+j-1}), a_{n+j}^{2n-1}),$ \\
then the n-ary hyperoperation $f$ is called associative. An $n$-ary hypergroupoid with the
associative $n$-ary hyperoperation is called an $n$-ary semihypergroup. 

An $n$-ary hypergroupoid $(G, f)$ in which the equation $y \in f(x_1^{i-1}, a_i, x_{ i+1}^n)$ has a solution $a_i \in G$
for every $x_1^{i-1}, x_{ i+1}^n,y \in G$ and $1 \leq i \leq n$, is called an $n$-ary quasihypergroup, when $(G, f)$ is an $n$-ary
semihypergroup, $(G, f)$ is called an $n$-ary hypergroup. 

An $n$-ary hypergroupoid $(G, f)$ is commutative if for all $ \sigma \in \mathbb{S}_n$, the group of all permutations of $\{1, 2, 3,..., n\}$, and for every $x_1^n \in G$ we have $f(x_1,..., x_n) = f(x_{\sigma(1)},..., x_{\sigma(n)})$.
If $x_1^n \in G$ then we denote $x_{\sigma(1)}^{\sigma(n)}$ as the $(x_{\sigma(1)},..., x_{\sigma(n)})$.

If $f$ is an $n$-ary hyperoperation and $t = l(n- 1) + 1$, then $t$-ary hyperoperation $f_{(l)}$ is given by
$f_{(l)}(x_1^{l(n-1)+1}) = f(f(..., f(f(x^n _1), x_{n+1}^{2n -1}),...), x_{(l-1)(n-1)+1}^{l(n-1)+1})$. 
\begin{definition}
\cite{d1} Let $(G, f)$ be an $n$-ary hypergroup and $H$ be a non-empty subset of $G$. $H$ is 
an $n$-ary subhypergroup of $(G, f)$, if $f(a^n _1) \subseteq H$ for $a^n_ 1 \in H$, and the equation $b \in f(b^{i-1}_1, x_i, b^n _{i+1})$ has a solution $x_i \in H$ for every $b^{i-1}_1, b^n _{i+1}, b \in H$ and $1 \leq i \leq n$.
An element $e \in G$ is said to be a scalar neutral element if $a = f(e^{(i-1)}, a, e^{(n-i)})$, for every $1 \leq i \leq n$ and
for every $a \in G$. 

An element $0$ of an $n$-ary semihypergroup $(G, g)$ is  a zero element if for each $a^n_2 \in G$ we have
$g(0, a^n _2) = g(a_2, 0, a^n_ 3) = ... = g(a^n _2, 0) = 0$.
If $0$ and $0^ \prime $ are two zero elements, then $0 = g(0^ \prime , 0^{(n-1)}) = 0 ^ \prime$ and so the zero element is unique. 
\end{definition}
\begin{definition}
\cite{l1} Let $(G, f)$ be an $n$-ary hypergroup. $(G, f)$ is called a canonical $n$-ary
hypergroup if\\
(1) there exists a unique $e \in G$, such that for every $a \in G, f(a, e^{(n-1)}) = a$;\\
(2) for all $a \in G$ there exists a unique $a^{-1} \in G$, such that $e \in f(a, a^{-1}, e^{(n-2)})$;\\
(3) if $a \in f(a^n _1)$, then for all $i$, we have $a_i \in f(a, a^{-1},..., a^{-1}_{ i-1}, a^{-1}_ {i+1},..., a^{-1}_ n)$.

We say that $e$ is the scalar identity of $(G, f)$ and $a^{-1}$ is the inverse of $a$. Notice that the inverse of $e$ is $e$.
\end{definition}
\begin{definition} \cite{d1}
A Krasner $(m, n)$-hyperring is an algebraic hyperstructure $(R, f, g)$, or simply $R$,  which satisfies the following axioms:

(1) $(R, f$) is a canonical $m$-ary hypergroup;

(2) $(R, g)$ is a $n$-ary semigroup;

(3) the $n$-ary operation $g$ is distributive with respect to the $m$-ary hyperoperation $f$ , i.e., for every $x^{i-1}_1 , x^n_{ i+1}, a^m_ 1 \in R$, and $1 \leq i \leq n$,
$g(x^{i-1}_1, f(a^m _1 ), x^n _{i+1}) = f(g(x^{i-1}_1, a_1, x^n_{ i+1}),..., g(x^{i-1}_1, a_m, x^n_{ i+1}))$;

(4) $0$ is a zero element (absorbing element) of the $n$-ary operation $g$, i.e., for every $a^n_ 2 \in R$ we have 
$g(0, a^n _2) = g(a_2, 0, a^n _3) = ... = g(a^n_ 2, 0) = 0$.
\end{definition}
A non-empty subset $S$ of $R$ is called a subhyperring of $R$ if $(S, f, g)$ is a Krasner $(m, n)$-hyperring. Let 
$I$ be a non-empty subset of $R$, we say that $I$ is a hyperideal of $(R, f, g)$ if $(I, f)$ is an $m$-ary subhypergroup
of $(R, f)$ and $g(a^{i-1}_1, I, a_{i+1}^n) \subseteq I$, for every $a^n _1 \in R$ and $1 \leq i \leq n$.
\section{Interval valued $(\alpha,\beta)$-fuzzy hyperideals}
In this section, we introduce the notion of n-ary interval valued $(\alpha, \beta)$-fuzzy hyperideal in a Krasner $(m,n)$-hyperring $\mathcal{R}$ where $\alpha, \beta \in \{\in, q, \in \vee q, \in \wedge q\}$. 
\begin{definition} \label{1}
An interval valued fuzzy set $\mathcal{A}$ of a Krasner $(m,n)$-hyperring $\mathcal{R}$ is said to be n-ary interval valued $(\alpha, \beta)$-fuzzy hyperideal of $\mathcal{R}$ where $\alpha, \beta \in \{\in, q, \in \vee q, \in \wedge q\}$ if for all $s_1^m, t_1^n, s \in (0,1]$ and $a_1^m, b_1^n,b \in \mathcal{R}$, the following conditions hold:
\begin{itemize} 
\item[\rm{(1)}] $F(a_1;
\tilde{s}_1)\alpha\mathcal{A},\dots, F(a_m;
\tilde{s}_m)\alpha\mathcal{A}$ impliy that $F(a;
rmin\{\tilde{s}_1,\dots,\tilde{s}_m\})\beta\mathcal{A}$, for all $a \in f(a_1^m)$,
\item[\rm{(2)}] $F(b;\tilde{s}) \alpha \mathcal{A}$ implies that $F(-b;\tilde{s}) \beta \mathcal{A}$,
\item[\rm{(3)}] $F(b_1;
\tilde{s}_1)\alpha\mathcal{A},\dots, F(b_n;
\tilde{s}_m)\alpha\mathcal{A}$ imply that $F(g(b_1^n);
\tilde{s})\beta\mathcal{A}$,
\end{itemize} 
\end{definition}
Notice that $\alpha=\in \wedge q$ in Definition \ref{1} should not be considered. Let $\tilde{\mu}_{\mathcal{A}}(a) \leq [0.5,0.5]$ for an interval valued fuzzy set $\mathcal{A}$ of $\mathcal{R}$ and for all $a \in \mathcal{R}$. Assume that $F(a;\tilde{s})\in \wedge q \mathcal{A}$ for $a \in \mathcal{R}$ and $s \in (0,1]$. This means $\tilde{\mu}_{\mathcal{A}}(a) \geq \tilde{s}$ and $\tilde{\mu}_{\mathcal{A}}(a) + \tilde{s} > [1,1]$. Then we have
\[[1,1] < \tilde{\mu}_{\mathcal{A}}(a) + \tilde{s} \leq \tilde{\mu}_{\mathcal{A}}(a) + \tilde{\mu}_{\mathcal{A}}(a)=2\tilde{\mu}_{\mathcal{A}}(a)\]
and so $\tilde{\mu}_{\mathcal{A}}(a) > [0.5,0.5]$. This follows that $\{F(a; \tilde{s}) \ \vert \ F(a; \tilde{s}) \in \wedge q \mathcal{A}\}=\varnothing$.
\begin{theorem} \label{2}
Let $\mathcal{A}$ be an n-ary interval valued $(\in, \in )$-fuzzy hyperideal of  a Krasner $(m,n)$-hyperring $\mathcal{R}$. Then $\mathcal{A}$ is an n-ary interval valued $(\in , \in \vee q)$-fuzzy hyperideal of $\mathcal{R}$.
\end{theorem}
\begin{proof}
The proof is straightforward.
\end{proof}
\begin{theorem} \label{3}
Let $\mathcal{A}$ be an n-ary interval valued $(\in \vee q, \in \vee q)$-fuzzy hyperideal of  a Krasner $(m,n)$-hyperring $\mathcal{R}$. Then $\mathcal{A}$ is an n-ary interval valued $(\in , \in \vee q)$-fuzzy hyperideal of $\mathcal{R}$.
\end{theorem}
\begin{proof}
Let $\mathcal{A}$ be an n-ary interval valued $(\in \vee q, \in \vee q)$-fuzzy hyperideal of $\mathcal{R}$. Let $F(a_1;\tilde{s}_1)\in\mathcal{A},\dots, F(a_m;\tilde{s}_m) \in \mathcal{A}$ for all $s_1^m \in (0,1]$ and $a_1^m \in \mathcal{R}$. Therefore $F(a_1;\tilde{s}_1)\in \vee q \mathcal{A},\dots, F(a_m;\tilde{s}_m) \in \vee q \mathcal{A}$. Since  $\mathcal{A}$ be an n-ary interval valued $(\in \vee q, \in \vee q)$-fuzzy hyperideal of $\mathcal{R}$, then $F(a;
rmin\{\tilde{s}_1,\dots,\tilde{s}_m\})\in \vee q\mathcal{A}$, for all $a \in f(a_1^m)$. The proofs of the other cases are similar.
\end{proof}
\begin{theorem} \label{4}
Let $I$ be a hyperideal of  a Krasner $(m,n)$-hyperring $\mathcal{R}$. Then the characteristic function $\chi_I$ of $I$ is an n-ary interval valued $(\in, \in)$-fuzzy hyperideal of $\mathcal{R}$.
\end{theorem}
\begin{proof}
Let $I$ be a hyperideal of  a Krasner $(m,n)$-hyperring $\mathcal{R}$. Let $F(a_1;\tilde{s}_1)\in\chi_I,\dots, F(a_m;\tilde{s}_m) \in \chi_I$ for all $s_1^m \in (0,1]$ and $a_1^m \in \mathcal{R}$. Then for $1 \leq i \leq m$, $\tilde{\chi}_I(a_i) \geq \tilde{s}_i> [0,0]$. Then we get $\tilde{\chi}_I(a_i)=[1,1]$ for all $1 \leq i \leq m$. This means $a_i \in \chi_I$ for all $1 \leq i \leq m$. Therefore $a \in \chi_I$ for all $a \in f(a_1^m)$. Thus $\tilde{\chi}_I(a)=[1,1] \geq rmin\{\tilde{s}_1,\dots, \tilde{s}_m\}$ which implies $F(a;rmin\{\tilde{s}_1,\dots, \tilde{s}_m\}) \in \chi_I$. The proofs of the  other conditions are similar. 
\end{proof}

Let $\mathcal{A}$ be an interval valued fuzzy set. The set $F(\mathcal{A};\tilde{s})=\{a \in \mathcal{R} \ \vert \ \tilde{\mu}_{\mathcal{A}}(a) \geq \tilde{s} \}$ is called the interval
valued level subset of $\mathcal{A}$. We say that an interval valued fuzzy set $\mathcal{A}$ of a Krasner $(m,n)$-hyperring $\mathcal{R}$ is proper if $\vert Im \mathcal{A} \vert \geq 2$. If two interval
valued fuzzy sets have the same family of interval valued level subsets, then they are said to be equivalent.
\begin{theorem} \label{7}
Let $\mathcal{R}$ is a Krasner $(m,n)$-hyperring containing some proper hyperideals and let $\mathcal{A}$ be an proper interval valued $(\in, \in)$-fuzzy hyperideal of $\mathcal{R}$ with $\vert Im \mathcal{A} \vert \geq 3$. Then $\mathcal{A}=\mathcal{A}_1 \cup \mathcal{A}_2$ such that  $\mathcal{A}_1$ and $\mathcal{A}_2$ are non-equivalent interval valued $(\in, \in)$-fuzzy hyperideals of $\mathcal{R}$.
\end{theorem}
\begin{proof}
The proof is similar to the proof of Theorem 3.7 in \cite{zhan}.
\end{proof}
\section{$n$-ary interval valued $(\in, \in \vee q)$-fuzzy hyperideals}
In this section, we first generalize the notion of fuzzy hyperideals to the notion of interval valued fuzzy hyperideals in a  Krasner $(m,n)$-hyperring. Then we study some fundamental aspects of 
the interval valued $(\in, \in \vee q)$-fuzzy hyperideals in a  Krasner $(m,n)$-hyperring.
\begin{definition} \label{8}
Let $\mathcal{A}$ be an interval valued fuzzy set of a Krasner $(m,n)$-hyperring $\mathcal{R}$. $\mathcal{A}$ refers to an n-ary  interval valued fuzzy hyperideal if for all $a_1^m, b_1^n,b \in \mathcal{R}$, the following conditions hold:
\begin{itemize} 
\item[\rm{(i)}]~ $rmin\{\tilde{\mu}_{\mathcal{A}}(a_1),\dots,\tilde{\mu}_{\mathcal{A}}(a_m)\} \leq rinf \{\tilde{\mu}_{\mathcal{A}}(a) \ \vert \ a \in f(a_1^m)\}$,
\item[\rm{(ii)}]~$\tilde{\mu}_{\mathcal{A}}(b) \leq \tilde{\mu}_{\mathcal{A}}(-b)$
\item[\rm{(iii)}]~$rmax\{\tilde{\mu}_{\mathcal{A}}(b_1),\dots,\tilde{\mu}_{\mathcal{A}}(b_n)\} \leq \tilde{\mu}_{\mathcal{A}}(g(b_1^n)).$
\end{itemize}
\end{definition}
The following is a direct consequence and can be proved easily and so the proof is omited.

\begin{theorem} \label{9}
Let $\mathcal{A}$ be an interval valued fuzzy set of a Krasner $(m,n)$-hyperring $\mathcal{R}$. Then $\mathcal{A}$ is an n-ary interval valued fuzzy hyperideal of $\mathcal{R}$ if and only if for each $[0,0] < \tilde{s} \leq [1,1]$, $F(\mathcal{A};\tilde{s}) (\neq \varnothing)$ is a hyperideal of $\mathcal{R}$.
\end{theorem} 
\begin{definition} \label{10}
Let $\mathcal{A}$ be an interval valued fuzzy set of a Krasner $(m,n)$-hyperring $\mathcal{R}$. $\mathcal{A}$ is called an n-ary interval valued $(\in, \in \vee q)$-fuzzy hyperideal of $\mathcal{R}$ if for all $s_1^m, t_1^n, t \in (0,1]$ and $a_1^m,b_1^n,b \in \mathcal{R}$
\begin{itemize} 
\item[\rm{(i1)}]~ $F(a_1;\tilde{s}_1) \in \mathcal{A}, \dots, F(a_m;\tilde{s}_m) \in \mathcal{A}$ imply $F(a;rmin\{\tilde{s}_1,\dots,\tilde{s}_m\})\in \vee q \mathcal{A}$, for all $a \in f(a_1^m)$,
\item[\rm{(ii1)}]~ $F(b;\tilde{t}) \in \mathcal{A}$ implies $F(-b;\tilde{t}) \in \vee q \mathcal{A}$
\item[\rm{(iii1)}]~$F(b_1;\tilde{t}_1) \in \mathcal{A}, \dots, F(b_n;\tilde{t}_n) \in \mathcal{A}$ imply $F(g(b_1^n);rmax\{\tilde{t}_1,\dots,\tilde{t}_n\})\in \vee q \mathcal{A}$.
\end{itemize}
\end{definition}
It is clear that every n-ary  interval valued fuzzy hyperideal of a Krasner $(m,n)$-hyperring is an  n-ary interval valued $(\in, \in \vee q)$-fuzzy hyperideal of $\mathcal{R}$. The following example shows that the inverse  is not true, in general.

\begin{example} 
The set $R=\{0,1, 2, 3\}$ with following 2-hyperoperation $"\oplus"$ is a canonical 2-ary hypergroup.

\hspace{1.5cm}
\begin{tabular}{c|c} 
$\oplus$ & $0$ \ \ \ \ \ \ \ $1$ \ \ \ \ \ \ \ $2$ \ \ \ \ \ \ \ $3$
\\ \hline 0 & $0$\ \ \ \ \ \ \ $1$\ \ \ \ \ \ \ \ \ $2$ \ \ \ \ \ \ \ $3$ 
\\ $1$ & $1$ \ \ \ \ \ \ \ $A$ \ \ \ \ \ \ \ $3$ \ \ \ \ \ \ $B$
\\ $2$ & $2$ \ \ \ \ \ \ \ $3$ \ \ \ \ \ \ \ $0$ \ \ \ \ \ \ \ $1$
\\ $3$ & $3$ \ \ \ \ \ \ \ $B$ \ \ \ \ \ \ \ $1$ \ \ \ \ \ \ $A$
\end{tabular}

In which $A=\{0,1\}$ and $B=\{2, 3\}$. Define a 4-ary operation $g$ on $R$ as follows:
\[ 
g(a_1^n)=\bigg{\{}
\begin{array}{lr}
2 & \text{if $a_1,a_2,a_3,a_4 \in B$}\\
0 & \text{otherwise}
\end{array} \]
It follows that $(R,\oplus,g)$ is a Krasner (2,4)-hyperring. Now, we define $\tilde{\mu}_{\mathcal{A}}(0)= \tilde{\mu}_{\mathcal{A}}(1)=[0.8,0.9]$, $\tilde{\mu}_{\mathcal{A}}(2)= [0.7,0.8]$ and $\tilde{\mu}_{\mathcal{A}}(3)=[0.6,0.7]$. Then it is easy to check that $\mathcal{A}$ is an interval valued $(\in,\in \vee,q)$-)-fuzzy hyperideal of $\mathcal{R}$.
\end{example}
In the following theorem, we present an equivalent condition for n-ary interval valued $(\in, \in \vee q)$-fuzzy hyperideals. 
\begin{theorem} \label{11}
Let $\mathcal{A}$ be an interval valued fuzzy set of a Krasner $(m,n)$-hyperring $\mathcal{R}$. Then $\mathcal{A}$ is an n-ary interval valued $(\in, \in \vee q)$-fuzzy hyperideal of $\mathcal{R}$ if and only if for all $a_1^m,b_1^n,b \in \mathcal{R}$, 
\begin{itemize} 
\item[\rm{(i2)}]~ $rmin\{\tilde{\mu}_{\mathcal{A}}(a_1),\dots,\tilde{\mu}_{\mathcal{A}}(a_m),[0.5,0.5]\} \leq rinf \{\tilde{\mu}_{\mathcal{A}}(a) \ \vert \ a \in f(a_1^m)\}$,
\item[\rm{(ii2)}]~$rmin\{\tilde{\mu}_{\mathcal{A}}(b),[0.5.0.5]\} \leq \tilde{\mu}_{\mathcal{A}}(-b)$,
\item[\rm{(iii2)}]~$rmax\{\tilde{\mu}_{\mathcal{A}}(b_1),\dots,\tilde{\mu}_{\mathcal{A}}(b_n),[0.5,0.5]\} \leq \tilde{\mu}_{\mathcal{A}}(g(b_1^n)).$
\end{itemize}
\end{theorem}
\begin{proof}
We need to show that  the conditions of Definition \ref{10} are equivalent to the  conditions (i2), (ii2) and (iii2), respectively.

$(i1) \Longrightarrow (i2)$: Let $a_1^m \in \mathcal{R}$. Then There exist two cases to be considered:
\begin{itemize} 
\item[\rm{(1)}]~ $rmin\{\tilde{\mu}_{\mathcal{A}}(a_1),\dots,\tilde{\mu}_{\mathcal{A}}(a_m)\} \leq [0.5,0.5]$,
\item[\rm{(2)}]~ $rmin\{\tilde{\mu}_{\mathcal{A}}(a_1),\dots,\tilde{\mu}_{\mathcal{A}}(a_m)\} > [0.5,0.5]$.
\end{itemize}
{\bf Case (1):} Let $\tilde{\mu}_{\mathcal{A}}(a) < rmin\{\tilde{\mu}_{\mathcal{A}}(a_1),\dots,\tilde{\mu}_{\mathcal{A}}(a_m),[0.5,0.5]\}$, for some $a \in f(a_1^m)$. This means $\tilde{\mu}_{\mathcal{A}}(a) < rmin\{\tilde{\mu}_{\mathcal{A}}(a_1),\dots,\tilde{\mu}_{\mathcal{A}}(a_m)\}$. Take $s \in (0,1]$ such that $\tilde{\mu}_{\mathcal{A}}(a) < \tilde{s} <rmin\{\tilde{\mu}_{\mathcal{A}}(a_1),\dots,\tilde{\mu}_{\mathcal{A}}(a_m)\}$. This implies that for all $1 \leq i \leq m$, $F(a_i;\tilde{s}) \in \mathcal{A}$ but $F(a;\tilde{s}) \overline{\in \vee q} \mathcal{A}$, a contradiction.

{\bf Case (2):} Suppose that there exists $a \in f(a_1^m)$ with $\tilde{\mu}_{\mathcal{A}}(a) <[0.5,0.5]$. From  $rmin\{\tilde{\mu}_{\mathcal{A}}(a_1),\dots,\tilde{\mu}_{\mathcal{A}}(a_m)\} > [0.5,0.5]$, it follows that $F(a_i;[0.5,0.5]) \in \mathcal{A}$ for all $1 \leq i \leq m$. In the other hand, $F(a;[0.5.0.5]) \overline{\in \vee q} \mathcal{A}$, a contradiction.

$(ii1) \Longrightarrow (ii2)$: Let $b \in \mathcal{R}$. Then we have the following  two cases:
\begin{itemize} 
\item[\rm{(1)}]~ $\tilde{\mu}_{\mathcal{A}}(b) \leq [0.5,0.5]$,
\item[\rm{(2)}]~ $\tilde{\mu}_{\mathcal{A}}(b) > [0.5,0.5]$.
\end{itemize}
{\bf Case (1):} Let us consider $\tilde{\mu}_{\mathcal{A}}(b)=\tilde{s} <[0.5,0.5]$ and $\tilde{\mu}_{\mathcal{A}}(-b)=\tilde{t} <\tilde{\mu}_{\mathcal{A}}(b)$. We take $r$ such that $\tilde{t}<\tilde{r}<\tilde{s}$ and $\tilde{t}+\tilde{r}<[0.5,0,5]$. Thus we get $F(b;\tilde{r}) \in \mathcal{A}$ but $F(-b;\tilde{r}) \overline{\in \vee q}\mathcal{A}$. This is a contradiction.

{\bf Case (2):} Suppose that $\tilde{\mu}_{\mathcal{A}}(b) \geq [0.5,0.5]$ and $rmin\{\tilde{\mu}_{\mathcal{A}}(b),[0.5.0.5]\} >\tilde{\mu}_{\mathcal{A}}(-b)$. Therefore we obtain $F(b;[0.5,0.5]) \in \mathcal{A}$ but $F(-b;[0.5,0.5]) \overline{\in \vee q} \mathcal{A}$, a contradiction.

$(iii1) \Longrightarrow (iii2)$ By using an argument similar to that in the proof of $(i1) \Longrightarrow (i2)$ one can easily complete the proof.

$(i2) \Longrightarrow (i1)$ Let $F(a_1;\tilde{s}_1) \in \mathcal{A}, \dots, F(a_m;\tilde{s}_m) \in \mathcal{A}$ for $a_1^m \in \mathcal{R}$ and $s_1^m \in (0,1]$. This means that for all $1 \leq i \leq m$ we get $\tilde{\mu}_{\mathcal{A}}(a_i) \geq \tilde{s}_i$. On the other hand, we have $rmin\{\tilde{s}_1,\dots,\tilde{s}_m,[0.5,0.5]\} \leq rmin\{\tilde{\mu}_{\mathcal{A}}(a_1),\dots,\tilde{\mu}_{\mathcal{A}}(a_m),[0.5,0.5]\} \leq \tilde{\mu}_{\mathcal{A}}(a)$ for each $a \in f(a_1^m)$. Now, if $rmin\{\tilde{s}_1,\dots,\tilde{s}_m\}>[0.5,0.5]$, then we have $\tilde{\mu}_{\mathcal{A}}(a) \geq [0.5,0.5]$ and so $\tilde{\mu}_{\mathcal{A}}(a)+rmin\{\tilde{s}_1,\dots,\tilde{s}_m\}>[1,1]$. Otherwise, we get $\tilde{\mu}_{\mathcal{A}}(a) \geq rmin\{\tilde{s}_1,\dots,\tilde{s}_m\}$ and then $F(a;rmin\{\tilde{s}_1,\dots,\tilde{s}_m\}) \in \vee q \mathcal{A}$ for each $a \in f(a_1^m)$.

$(ii2) \Longrightarrow (ii1)$ Suppose that $F(b;\tilde{s}) \in \mathcal{A}$ for some $b \in \mathcal{R}$ and $s \in (0,1]$. This means $\tilde{\mu}_{\mathcal{A}}(b) \geq \tilde{s}$. Thus, $rmin\{\tilde{s},[0.5,0.5])\} \leq rmin\{\tilde{\mu}_{\mathcal{A}}(b),[0.5,0.5])\}  \leq \tilde{\mu}_{\mathcal{A}}(-b)$. Now, if $\tilde{s} \leq [0.5,0.5]$, then $\tilde{\mu}_{\mathcal{A}}(-b) \geq \tilde{s}$ and if $\tilde{s} \geq [0.5,0.5]$, then we get $\tilde{\mu}_{\mathcal{A}}(-b) \geq  [0.5,0.5]$.

$(iii2) \Longrightarrow (iii1)$ This can be proved in a very similar manner to the way in which $(i2) \Longrightarrow (i1)$ was proved. 
\end{proof}

\begin{theorem} \label{5}
Let $I$ be a subset of a Krasner $(m,n)$-hyperring $\mathcal{R}$. $I$ is a hyperideal of $\mathcal{R}$ if and only if $\chi_I$ is an n-ary interval valued $(\in, \in \vee q)$-fuzzy hyperideal of $\mathcal{R}$.
\end{theorem}
\begin{proof}
$\Longrightarrow$ Let $I$ is a hyperideal of $\mathcal{R}$. By Theorem \ref{4}, we conclude that $\chi_I$ is an n-ary interval valued $(\in, \in )$-fuzzy hyperideal of $\mathcal{R}$. Thus $\chi_I$ is an n-ary interval valued $(\in, \in \vee q)$-fuzzy hyperideal of $\mathcal{R}$, by Theorem \ref{2}.\\
$\Longleftarrow$ Suppose that $\chi_I$ is an n-ary interval valued $(\in, \in \vee q)$-fuzzy hyperideal of $\mathcal{R}$. Let $a_1^m \in I$. This means $F(a_i; [1,1]) \in \chi_I$ for all $1 \leq i \leq m$. Thus $F(a; [1,1])=F(a; rmin \{\underbrace{[1,1],\dots,[1,1]}_m\}) \in \vee q \chi_I$, for all $a \in f(a_1^m)$. Therefore $\tilde{\mu}_I(a)>[0,0]$ for all $a \in f(a_1^m)$ which implies $f(a_1^m) \subseteq I$. Assume that $a \in I$. Therefore $F(a;[1,1]) \in \chi_I$ which means $F(-a;[1,1]) \in \vee q \chi_I$. Hence $\tilde{\chi}_I(-a)>[0,0]$. Consequently, we get $-a \in I$. Now we suppose that $b_1^n \in \mathcal{R}$ and $b_i \in I$ for some $1 \leq i \leq n$. Then we get $F(b_i;[1,1]) \in \chi_I$. From 
$rmax\{\tilde{\chi}_I(b_1),\dots,\tilde{\chi}_I(b_{i-1}),\tilde{\chi}_I(b_i),\tilde{\chi}_I(b_{i+1}),\dots,\tilde{\chi}_I(b_n),[0.5,0.5]\} \leq \tilde{\chi}_I(g(b_1^n))$, it follows that  $F(g(b_1^n);[1,1]) \in \chi_I$. Thus $g(b_1^{i-1},b_i,b_{i+1}^n) \in I$. Hence $I$ is a hyperideal of $\mathcal{R}$.
\end{proof}
In the following theorem, we examine  n-ary interval valued $(\in, \in \vee q)$-fuzzy hyperideals by level subsets.
\begin{theorem} \label{12}
If $\mathcal{A}$ is an n-ary interval valued $(\in, \in \vee q)$-fuzzy hyperideal of  a Krasner $(m,n)$-hyperring $\mathcal{R}$, then $F(\mathcal{A};\tilde{s})$ is an
empty set or a hyperideal of $\mathcal{R}$ for every $[0,0] < \tilde{s} \leq [0.5,0.5]$.
\end{theorem}
\begin{proof}
Suppose that $\mathcal{A}$ is an n-ary interval valued $(\in, \in \vee q)$-fuzzy hyperideal of $\mathcal{R}$ and $[0,0] < \tilde{s} \leq [0.5,0.5]$. Let $a_1^m \in F(\mathcal{A};\tilde{s})$. This means that $\tilde{\mu}_{\mathcal{A}}(a_i) \geq \tilde{s}$ for all $1 \leq i \leq m$ and so $\tilde{s}=rmin\{\tilde{s},[0.5,0.5]\} \leq rmin\{\tilde{\mu}_{\mathcal{A}}(a_1),\dots, \tilde{\mu}_{\mathcal{A}}(a_m),[0.5,0.5]\} \leq rinf\{\tilde{\mu}_{\mathcal{A}}(a) \ \vert \ a \in f(a_1^m) \}$. Thus $f(a_1^m) \subseteq F(\mathcal{A};\tilde{s})$. Let $b \in F(\mathcal{A};\tilde{s})$. Then $\tilde{\mu}_{\mathcal{A}}(b) \geq \tilde{s}$. Now we get $\tilde{\mu}_{\mathcal{A}}(-b) \geq rmin\{\tilde{\mu}_{\mathcal{A}}(b),[0.5,0.5]\}=rmin\{\tilde{s},[0.5,0.5]\}=\tilde{s}$ which implies $-b \in F(\mathcal{A};\tilde{s})$. Moreover, let $b_1^n \in \mathcal{R}$ such that $b_i \in F(\mathcal{A};\tilde{s})$. Then $\tilde{\mu}_{\mathcal{A}}(g(b_1^n)) \geq rmax\{\tilde{\mu}_{\mathcal{A}}(b_1),\dots, \tilde{\mu}_{\mathcal{A}}(b_{i-1}),\tilde{\mu}_{\mathcal{A}}(b_i),\tilde{\mu}_{\mathcal{A}}(b_{i+1}),\dots,\tilde{\mu}_{\mathcal{A}}(b_n),[0.5,0.5]\} \geq \tilde{s}$. Hence we conclude that $g(b_1^n) \in F(\mathcal{A};\tilde{s})$. Consequently, $F(\mathcal{A};\tilde{s})$ is  a hyperideal of $\mathcal{R}$.
\end{proof}
In the next theorem, the subsets are discussed where $[0.5,0.5] < \tilde{s} \leq [1,1]$.
\begin{theorem} \label{13}
Let $\mathcal{A}$ be an interval valued fuzzy set of a Krasner $(m,n)$-hyperring $\mathcal{R}$. Then  $F(\mathcal{A};
\tilde{s}) (\neq \varnothing)$ is a hyperideal of $\mathcal{R}$ for every $[0.5,0.5] < \tilde{s} \leq [1,1]$ if
and only if for all $a_1^m,b_1^n,b \in \mathcal{R}$, 
\begin{itemize} 
\item[\rm{(1)}]~ $rmin\{\tilde{\mu}_{\mathcal{A}}(a_1),\dots,\tilde{\mu}_{\mathcal{A}}(a_m)\} \leq rinf \{rmax\{\tilde{\mu}_{\mathcal{A}}(a),[0.5,0.5]\} \ \vert \ a \in f(a_1^m)\}$,
\item[\rm{(2)}]~$\tilde{\mu}_{\mathcal{A}}(b) \leq rmax\{\tilde{\mu}_{\mathcal{A}}(-b),[0.5.0.5]\}$,
\item[\rm{(3)}]~$rmax\{\tilde{\mu}_{\mathcal{A}}(b_1),\dots,\tilde{\mu}_{\mathcal{A}}(b_n)\} \leq rmax\{\tilde{\mu}_{\mathcal{A}}(g(b_1^n)),[0.5,0.5]\}.$
\end{itemize}
\end{theorem}
\begin{proof}
$\Longrightarrow$ Suppose that $F(\mathcal{A};
\tilde{s}) (\neq \varnothing)$ is a hyperideal of $\mathcal{R}$ for every $[0.5,0.5] < \tilde{s} \leq [1,1]$. (1) Let $rmax\{\tilde{\mu}_{\mathcal{A}}(a),[0.5,0.5]\}<rmin\{\tilde{\mu}_{\mathcal{A}}(a_1),\dots,\tilde{\mu}_{\mathcal{A}}(a_m)\}=\tilde{s}$      for some $a_1^m \in \mathcal{R}$ such that $a \in f(a_1^m)$. This implies that  $[0.5,0.5]<\tilde{s}\leq [1,1]$ and $a_1^m \in F(\mathcal{A};
\tilde{s})$. Since $F(\mathcal{A};
\tilde{s})$ is a hyperideal of $\mathcal{R}$ and $a_1^m \in F(\mathcal{A};
\tilde{s})$, we get $f(a_1^m) \subseteq F(\mathcal{A};
\tilde{s}) $. This means that $\tilde{\mu}_{\mathcal{A}}(a) \geq \tilde{s}$ for each $a \in f(a_1^m)$, a contradiction. Thus $rmin\{\tilde{\mu}_{\mathcal{A}}(a_1),\dots,\tilde{\mu}_{\mathcal{A}}(a_m)\} \leq rinf \{rmax\{\tilde{\mu}_{\mathcal{A}}(a),[0.5,0.5]\} \ \vert \ a \in f(a_1^m)\}$ for all $a_1^m \in \mathcal{R}$. (2) Suppose that $\tilde{s}=\tilde{\mu}_{\mathcal{A}}(b) \geq rmax\{\tilde{\mu}_{\mathcal{A}}(-b),[0.5.0.5]\}$, for some $b \in \mathcal{R}$. This implies that $[0.5,0.5]<\tilde{s}\leq [1,1]$ and $b \in F(\mathcal{A};\tilde{s})$. Since $F(\mathcal{A};
\tilde{s})$ is a hyperideal of $\mathcal{R}$ and $b \in F(\mathcal{A};\tilde{s})$, we conclude that $-b \in F(\mathcal{A};
\tilde{s})$ which means $\tilde{\mu}_{\mathcal{A}}(-b) \geq \tilde{s}$, a contradiction. Hence $\tilde{\mu}_{\mathcal{A}}(b) \leq rmax\{\tilde{\mu}_{\mathcal{A}}(-b),[0.5.0.5]\}$ for $b \in \mathcal{R}$. (3) By a similar argument to that of (1), we can prove $rmax\{\tilde{\mu}_{\mathcal{A}}(b_1),\dots,\tilde{\mu}_{\mathcal{A}}(b_n)\} \leq rmax\{\tilde{\mu}_{\mathcal{A}}(g(b_1^n)),[0.5,0.5]\}$ for $b_1^n \in \mathcal{R}$.\\
$\Longleftarrow$ Suppose that $a_1^m \in F(\mathcal{A};
\tilde{s})$ with $[0.5,0.5] < \tilde{s} \leq [1,1]$. Then $[0.5,0.5]<\tilde{s}\leq rmin\{\tilde{\mu}_{\mathcal{A}}(a_1),\dots,\tilde{\mu}_{\mathcal{A}}(a_m)\} \leq rinf \{rmax\{\tilde{\mu}_{\mathcal{A}}(a),[0.5,0.5]\} \ \vert \ a \in f(a_1^m)\}$ and so $\tilde{s} \leq rinf \{\tilde{\mu}_{\mathcal{A}}(a) \ \vert \ a \in f(a_1^m)\}$. Thus $f(a_1^m) \subseteq F(\mathcal{A};
\tilde{s})$. Also, let $b \in F(\mathcal{A};
\tilde{s})$. Then we get $[0.5,0.5] < \tilde{s} \leq \tilde{\mu}_{\mathcal{A}}(b) \leq rmax\{\tilde{\mu}_{\mathcal{A}}(-b),[0.5.0.5]\}$. Therefore $\tilde{s} \leq \tilde{\mu}_{\mathcal{A}}(-b)$ which means $-b \in F(\mathcal{A};
\tilde{s})$. Now, let $b_1^n \in \mathcal{R}$ and $b_i \in F(\mathcal{A};
\tilde{s})$. Then we obtain  $[0.5,0.5] < \tilde{s} \leq rmax\{\tilde{\mu}_{\mathcal{A}}(b_1),\dots,\tilde{\mu}_{\mathcal{A}}(b_{i-1}),\tilde{\mu}_{\mathcal{A}}(b_i),\tilde{\mu}_{\mathcal{A}}(b_{i+1}),\dots\tilde{\mu}_{\mathcal{A}}(b_n)\} \leq rmax\{\tilde{\mu}_{\mathcal{A}}(g(b_1^n)),[0.5,0.5]\}$. Hence $\tilde{s} \leq \tilde{\mu}_{\mathcal{A}}(g(b_1^n))$ and so $g(b_1^n) \in F(\mathcal{A};
\tilde{s})$. Consequently, $F(\mathcal{A};
\tilde{s})$ is a hyperideal of $\mathcal{R}$ for every $[0.5,0.5] < \tilde{s} \leq [1,1]$.
\end{proof}
\begin{definition} \label{14}
Let $\mathcal{A}$ be an interval valued fuzzy set of a Krasner $(m,n)$-hyperring $\mathcal{R}$ and $s_1,s_2 \in [0,1]$ with $\tilde{s}_1 < \tilde{s}_2$. $\mathcal{A}$ refers to an n-ary interval valued fuzzy hyperideal with thresholds $(\tilde{s}_1 , \tilde{s}_2)$ of $\mathcal{R}$ if for all $a_1^m,b_1^n,b \in \mathcal{R}$, $\mathcal{A}$ satisfies the following conditions:
\begin{itemize} 
\item[\rm{(1)}]~ $rmin\{\tilde{\mu}_{\mathcal{A}}(a_1),\dots,\tilde{\mu}_{\mathcal{A}}(a_m),\tilde{s}_2\} \leq rinf \{rmax\{\tilde{\mu}_{\mathcal{A}}(a),\tilde{s}_1\} \ \vert \ a \in f(a_1^m)\}$,
\item[\rm{(2)}]~$rmin\{\tilde{\mu}_{\mathcal{A}}(b),\tilde{s}_2\} \leq rmax\{\tilde{\mu}_{\mathcal{A}}(-b),\tilde{s}_1\}$,
\item[\rm{(3)}]~$rmax\{\tilde{\mu}_{\mathcal{A}}(b_1),\dots,\tilde{\mu}_{\mathcal{A}}(b_n),\tilde{s}_2\} \leq rmax\{\tilde{\mu}_{\mathcal{A}}(g(b_1^n)),\tilde{s}_1\}.$
\end{itemize}
\end{definition}
Suppose that $\mathcal{A}$  is an n-ary interval valued fuzzy hyperideal with thresholds  of a Krasner $(m,n)$-hyperring $\mathcal{R}$. Let us consider $\tilde{s}_1=[0,0]$ and $\tilde{s}_2=[1,1]$. Then $\mathcal{A}$ is an ordinary interval valued fuzzy hyperideal. Moreover, if we consider $\tilde{s}_1=[0,0]$ and $\tilde{s}_2=[0.5,0.5]$, then  $\mathcal{A}$ is an n-ary interval valued $(\in, \in \vee q)$-fuzzy hyperideal. 
\begin{theorem} \label{15}
Let $\mathcal{A}$ be an interval valued fuzzy set of a Krasner $(m,n)$-hyperring $\mathcal{R}$. Then $\mathcal{A}$ is an n-ary interval valued fuzzy hyperideal with thresholds $(\tilde{s}_1 , \tilde{s}_2)$ of $\mathcal{R}$ if and only if $F(\mathcal{A};\tilde{s})$ is a hyperideal of $\mathcal{R}$ for every $\tilde{s}_1 <\tilde{s} \leq \tilde{s}_2$.
\end{theorem}
\begin{proof}
$\Longrightarrow$ Let $a_1^m \in F(\mathcal{A};\tilde{s})$. Thereby we have $\tilde{\mu}_{\mathcal{A}}(a_i) \geq \tilde{s}$ for all $1 \leq i \leq m$. Then we deduce $\tilde{s}_1<\tilde{s} \leq rmin\{\tilde{s},\tilde{s}_2\} \leq rmin\{\tilde{\mu}_{\mathcal{A}}(a_1),\dots,\tilde{\mu}_{\mathcal{A}}(a_m),\tilde{s}_2\}$ and so
$rinf \{rmax\{\tilde{\mu}_{\mathcal{A}}(a),\tilde{s}_1\} \ \vert \ a \in f(a_1^m)\} \geq \tilde{s}$. This means that $rmax\{\tilde{\mu}_{\mathcal{A}}(a),\tilde{s}_1\} > \tilde{s}$ for all $a \in f(a_1^m)$. Then we get $\tilde{\mu}_{\mathcal{A}}(a)>\tilde{s}$ which implies $a \in F(\mathcal{A};\tilde{s})$. Therefore $f(a_1^m) \subseteq F(\mathcal{A};\tilde{s})$. Also, we assume that $b \in F(\mathcal{A};\tilde{s})$. Then $rmax\{\tilde{\mu}_{\mathcal{A}}(-b),\tilde{s}_1\} \geq rmin\{\tilde{\mu}_{\mathcal{A}}(b),\tilde{s}_2\} \geq \tilde{s}>\tilde{s}_1$. Thereby we have $\tilde{\mu}_{\mathcal{A}}(-b) \geq \tilde{s}$ which means $-b \in F(\mathcal{A};\tilde{s})$. Now, we consider $b_1^n \in \mathcal{R}$ and $b_i \in F(\mathcal{A};\tilde{s})$. Then  $\tilde{\mu}_{\mathcal{A}}(b_i) \geq \tilde{s}$. Hence we obtain $rmax\{\tilde{\mu}_{\mathcal{A}}(b_1),\dots,\tilde{\mu}_{\mathcal{A}}(b_n),\tilde{s}_2\} \geq rmax\{\tilde{s},\tilde{s}_2\} \geq \tilde{s}>\tilde{s}_1$. Then we get $rmax\{\tilde{\mu}_{\mathcal{A}}(g(b_1^n)),\tilde{s}_1\} \geq \tilde{s}$. Therefore $\tilde{\mu}_{\mathcal{A}}(g(b_1^n)) \geq \tilde{s}$ which means $g(b_1^n) \in F(\mathcal{A};\tilde{s})$. Thus $F(\mathcal{A};\tilde{s})$ is a hyperideal of $\mathcal{R}$ for every $\tilde{s}_1 <\tilde{s} \leq \tilde{s}_2$.\\
$\Longleftarrow$ Let $rmax\{\tilde{\mu}_{\mathcal{A}}(a),\tilde{s}_1\}<rmin\{\tilde{\mu}_{\mathcal{A}}(a_1),\dots, \tilde{\mu}_{\mathcal{A}}(a_m),\tilde{s}_2\}=\tilde{s}$ for some $a \in f(a_1^m)$ such that $a_1^m \in \mathcal{R}$. This means $\tilde{s}_1 <\tilde{s}\leq \tilde{s}_2$ and $a_1^m \in F(\mathcal{A};\tilde{s})$. Then $f(a_1^m) \subseteq F(\mathcal{A};\tilde{s})$, as $F(\mathcal{A};\tilde{s})$ is a hyperideal of $\mathcal{R}$. Therefore we get $\tilde{\mu}_{\mathcal{A}}(a) \geq \tilde{s}$ for all $a \in f(a_1^m)$, a contradiction. Then for all $a_1^m \in  \mathcal{R}$, $rmin\{\tilde{\mu}_{\mathcal{A}}(a_1),\dots,\tilde{\mu}_{\mathcal{A}}(a_m),\tilde{s}_2\} \leq rmax\{\tilde{\mu}_{\mathcal{A}}(a),\tilde{s}_1\}$. Also, put $\tilde{s}=rmax\{\tilde{\mu}_{\mathcal{A}}(-b),\tilde{s}_1\}$ for some $b \in \mathcal{R}$. Suppose that  $rmin\{\tilde{\mu}_{\mathcal{A}}(b),\tilde{s}_2\} >\tilde{s}$. Hence $b \in  F(\mathcal{A};\tilde{s})$ which implies $-b \in  F(\mathcal{A};\tilde{s})$ and so $\tilde{\mu}_{\mathcal{A}}(-b) \geq \tilde{s}$, a contradiction. Thus $rmin\{\tilde{\mu}_{\mathcal{A}}(b),\tilde{s}_2\} \leq rmax\{\tilde{\mu}_{\mathcal{A}}(-b),\tilde{s}_1\}$ for $b \in \mathcal{R}$. Now, put $\tilde{s}=rmax\{\tilde{\mu}_{\mathcal{A}}(b_1),\dots,\tilde{\mu}_{\mathcal{A}}(b_n),\tilde{s}_2\}$ for some $b_1^n \in \mathcal{R}$. We assume that $rmax\{\tilde{\mu}_{\mathcal{A}}(g(b_1^n)),\tilde{s}_1\}<\tilde{s}$. Then we get $\tilde{s}_1 <\tilde{s} \leq \tilde{s}_2$ and  $g(b_1^n) \in F(\mathcal{A};\tilde{s})$. Therefore $\tilde{\mu}_{\mathcal{A}}(g(b_1^n)) \geq \tilde{s}$, a contradiction. So $rmax\{\tilde{\mu}_{\mathcal{A}}(b_1),\dots,\tilde{\mu}_{\mathcal{A}}(b_n),\tilde{s}_2\} \leq rmax\{\tilde{\mu}_{\mathcal{A}}(g(b_1^n)),\tilde{s}_1\}$, for every $b_1^n \in \mathcal{R}$. Consequently, $\mathcal{A}$ is an n-ary interval valued fuzzy hyperideal with thresholds $(\tilde{s}_1 , \tilde{s}_2)$ of $\mathcal{R}$.
\end{proof}
\section{Implication-based interval valued fuzzy hyperideals of a Krasner $(m,n)$-hyperring}
Logic is a study of language in arguments and
persuasion. We can  use it to judge the correctness of a
chain of reasoning in a mathematical proof. Fuzzy logic is a generalization of set theoretic variables in terms of the linguistic variable truth. By
using extension principal some operators like $\vee, \wedge, \rightharpoondown, \longrightarrow$ can be applied in fuzzy
logic. In the fuzzy logic, $[P]$ denotes  the truth value of fuzzy proposition $P$. In the following,  a correspondence between set-theoretical notions  and fuzzy logic is shown.\\

$ [a \in \mathcal{A}]=\mathcal{A}(a);$

$[a \notin \mathcal{A}]=1-\mathcal{A}(a);$

$[P \vee Q]=max\{[P],[Q]\};$

$[P \wedge Q]=min\{[P],[Q]\};$

$[\forall a \ P(a)]=inf[P(a)];$

$[P \longrightarrow Q]=min\{1,1-[P]+[Q]\};$

$\models P$ if and only if $[P]=1$ for all valuations.

  $\vspace{0.2cm}$
 
Of course, various implication operators can been deﬁned. In the following, we give a selection of the most important multi-valued
implications, where $\alpha$ means the degree
of membership of the premise, $\beta$ the respective
values for the consequence and $I$ the resulting degree of truth for the implication.
 $\vspace{0.5cm}$
 
 Early Zadeh: $\hspace{0.5cm}I_m(\alpha,\beta)=max\{1-\alpha, min\{\alpha, \beta\}\},$

  $\vspace{0.3cm}$

 Lukasiewicz: $\hspace{0.5cm}I_a(\alpha,\beta)=min\{1,1-\alpha+\beta\},$
 
  $\vspace{0.3cm}$
 
Standard Star (Godel):  $\hspace{0.5cm}
 I_g(\alpha,\beta)=\left\{
 \begin{array}{lr}
 1&\text{if $ \alpha \leq \beta $,}\\
 \beta&\text{otherwise,}
 \end{array} \right.$
 
  $\vspace{0.3cm}$
 
  Contraposition of Godel: $
\hspace{0.5cm} I_{cg}(\alpha,\beta)=\left\{
 \begin{array}{lr}
 1&\text{if $ \alpha \leq \beta $,}\\
 1-\alpha &\text{otherwise,}
 \end{array} \right.$
 
  $\vspace{0.3cm}$
 
Gaines–Rescher: $\hspace{0.5cm}
 I_{gr}(\alpha,\beta)=\left\{
 \begin{array}{lr}
 1&\text{if $ \alpha \leq \beta $,}\\
 0 &\text{otherwise,}
 \end{array} \right.$
 
  $\vspace{0.3cm}$
 
  Kleene–Dienes: $\hspace{0.5cm}I_b(\alpha,\beta)=max\{1-\alpha, \beta\},$
 
  $\vspace{0.3cm}$
 
 Goguen: $
\hspace{0.5cm} I_{gg}(\alpha,\beta)=\left\{
 \begin{array}{lr}
 1&\text{if $ \alpha \leq \beta $,}\\
 \frac{\beta}{\alpha} &\text{if $\alpha > \beta $}
 \end{array} \right.$
 
 $\vspace{0.3cm}$
 
 In the following definition, we considered the definition of  implicative operator in the Lukasiewicz system of continuous-valued
logic.
\begin{definition} \label{16}
Let $\mathcal{A}$ be an  interval valued fuzzy set of a Krasner $(m,n)$-hyperring $\mathcal{R}$. $\mathcal{A}$ refers to an n-ary interval valued fuzzifying hyperideal of $\mathcal{R}$ if it satisfies:
\begin{itemize} 
\item[\rm{(1)}]~ $\models[rmin\{[a_1 \in \mathcal{A}], \dots, [a_m \in \mathcal{A}]\} \longrightarrow [\forall a \in f(a_1^m), \ a \in \mathcal{A}]]$, for all $a_1^m \in \mathcal{R}$, 
\item[\rm{(2)}]~ $\models [[b \in \mathcal{A}] \longrightarrow [-b \in \mathcal{A}]]$, for each $b \in \mathcal{R}$,
\item[\rm{(3)}]~ $\models[rmax\{[b_1 \in \mathcal{A}], \dots, [b_n \in \mathcal{A}]\} \longrightarrow [g(b_1^n) \in \mathcal{A}]]$, for all $b_1^n \in \mathcal{R}.$
\end{itemize}
\end{definition}
It is clear that an interval valued fuzzifying hyperideal is  an
ordinary n-ary interval valued fuzzy hyperideal. We  have the notion of interval valued $\tilde{t}$-tautology. In fact, $\models_{\tilde{t}}P$
if and only if $[P] \geq \tilde{t}$ for all valuations (see \cite{zhan}). Now, we give the following definition.
\begin{definition}
Let $\mathcal{A}$ be an  interval valued fuzzy set of a Krasner $(m,n)$-hyperring $\mathcal{R}$ and $t \in (0,1]$. Then $\mathcal{A}$ is said to be an n-ary $\tilde{t}$-implication-based interval valued 
fuzzy hyperideal of $\mathcal{R}$ if the following conditions hold:
\begin{itemize} 
\item[\rm{(1)}]~ $\models_{\tilde{t}}[rmin\{[a_1 \in \mathcal{A}], \dots, [a_m \in \mathcal{A}]\} \longrightarrow [\forall a \in f(a_1^m), \ a \in \mathcal{A}]]$, for all $a_1^m \in \mathcal{R}$, 
\item[\rm{(2)}]~ $\models_{\tilde{t}} [[b \in \mathcal{A}] \longrightarrow [-b \in \mathcal{A}]]$, for each $b \in \mathcal{R}$,
\item[\rm{(3)}]~ $\models_{\tilde{t}}[rmax\{[b_1 \in \mathcal{A}], \dots, [b_n \in \mathcal{A}]\} \longrightarrow [g(b_1^n) \in \mathcal{A}]]$, for all $b_1^n \in \mathcal{R}.$
\end{itemize}
\end{definition}
\begin{corollary}
Let $\mathcal{A}$ be an  interval valued fuzzy set of a Krasner $(m,n)$-hyperring $\mathcal{R}$ and $I$ be an implicative operator.  Then $\mathcal{A}$ is  an n-ary $\tilde{t}$-implication-based interval valued  fuzzy hyperideal of $\mathcal{R}$ for some $t \in (0,1]$ if and only if 
\begin{itemize} 
\item[\rm{(1)}]~ for any $a_1^m \in \mathcal{R}$, $I(rmin\{\tilde{\mu}_{\mathcal{A}}(a_1),\dots,\tilde{\mu}_{\mathcal{A}}(a_m), rinf \{rmax\{\tilde{\mu}_{\mathcal{A}}(a),\tilde{s}_1\} \ \vert \ a \in f(a_1^m)\} \geq \tilde{t}$,
\item[\rm{(2)}]~for any $b \in \mathcal{R}$, $I(\tilde{\mu}_{\mathcal{A}}(b),\tilde{\mu}_{\mathcal{A}}(-b)\} \geq \tilde{t}$,
\item[\rm{(3)}]~for any $b_1^n \in \mathcal{R}$,$I(rmax\{\tilde{\mu}_{\mathcal{A}}(b_1),\dots,\tilde{\mu}_{\mathcal{A}}(b_n),\tilde{\mu}_{\mathcal{A}}(g(b_1^n)) \} \geq \tilde{t}.$
\end{itemize}
\end{corollary}
\begin{theorem}
(1) Suppose that $I=I_{gr}$. Then $\mathcal{A}$ is an n-ary $[0.5,0.5]$-implication-based interval valued fuzzy hyperideal of $\mathcal{R}$ if and only if $\mathcal{A}$ is an n-ary 
interval valued fuzzy hyperideal with thresholds $(\tilde{s}=[0,0] ,\tilde{r}=[1,1])$ of $\mathcal{R}$.

(2) Suppose that $I=I_{g}$. Then $\mathcal{A}$ is an n-ary $[0.5,0.5]$-implication-based interval valued fuzzy hyperideal of $\mathcal{R}$ if and only if $\mathcal{A}$ is an n-ary 
interval valued fuzzy hyperideal with thresholds $(\tilde{s}=[0,0],\tilde{r}=[0.5,0.5])$ of $\mathcal{R}$.

(3) Suppose that $I=I_{cg}$. Then $\mathcal{A}$ is an n-ary $[0.5,0.5]$-implication-based interval valued fuzzy hyperideal of $\mathcal{R}$ if and only if $\mathcal{A}$ is an n-ary 
interval valued fuzzy hyperideal with thresholds $(\tilde{s}=[0.5,0.5],\tilde{r}=[1,1])$ of $\mathcal{R}$.
\end{theorem}
\begin{proof}
These can be proved by using definitions. 
\end{proof}
\begin{corollary}
(1) Let $I=I_{gr}$. Then $\mathcal{A}$ is an n-ary $[0.5,0.5]$-implication-based interval valued fuzzy hyperideal of $\mathcal{R}$ if and only if $\mathcal{A}$ is an ordinary n-ary interval valued fuzzy hyperideal  of $\mathcal{R}$.

(2) Let $I=I_{g}$. Then $\mathcal{A}$ is an n-ary $[0.5,0.5]$-implication-based interval valued fuzzy hyperideal of $\mathcal{R}$ if and only if $\mathcal{A}$ is an n-ary 
interval valued $(\in, \in \vee q)$-fuzzy hyperideal of $\mathcal{R}$.

(3) Let $I=I_{cg}$. Then $\mathcal{A}$ is an n-ary $[0.5,0.5]$-implication-based interval valued fuzzy hyperideal of $\mathcal{R}$ if and only if $\mathcal{A}$ is an n-ary 
interval valued $(\overline{\in}, \overline{\in \vee q})$-fuzzy hyperideal of $\mathcal{R}$.
\end{corollary}
\section{Conclusion}
Since hyperstructure theory was introduced by Marty in 1934, the idea has been investigated by many researches in the following decades. 
 In this paper, we introduced and characterized  the 
interval valued $(\alpha,\beta)$-fuzzy hyperideals of a Krasner $(m,n)$-hyperring, in which special attention was concentrated on the interval valued $(\in ,\in \vee q)$-fuzzy hyperideals. The consequences given in this paper can hopefully provide more realization into and
a full cognition of algebraic hyperstructures  and fuzzy set theory.


\end{document}